\documentclass[11pt]{amsart}
\usepackage{graphicx}
\usepackage{tikz}
\usepackage{latexsym}
\usepackage{amsfonts,amsmath,amssymb}
\usepackage{url}
\newtheorem{theorem}{Theorem}[section]								
\newtheorem{lemma}[theorem]{Lemma}
\newtheorem{definition}[theorem]{Definition}
\newtheorem{corollary}[theorem]{Corollary}
\newtheorem{example}[theorem]{Example}

\newtheorem{question}[theorem]{Question}

\usepackage{color}

\def\part{\partial}

\def\b1{\bold 1}

\newcommand{\beq}{\begin{equation}}
\newcommand{\eeq}{\end{equation}}

\theoremstyle{remark}

\numberwithin{equation}{section}
\date{\today}

\begin{document}

\title[A Bijective Proof of an Unbalanced Wilf-Equivalence]{A Bijective Proof of an Unbalanced Wilf-Equivalence}
\author {Jensen Bridges and Michael Waite}
\address{Department of Mathematics, University of Florida, Gainesville, FL 32601}
\email{jensenbridges@ufl.edu, michael.waite@ufl.edu}

\begin{abstract} 
We find a bijection to prove that the set of patterns $\{3412, 4321\}$ is Wilf-equivalent to the set of patterns $\{3412,4231,45321,54312\}$. This provides a bijective proof of an identity recently shown by Azam and Richmond.
\end{abstract}

\maketitle

\section{Introduction}

Let $Av_n(q)$ be the set of permutations avoiding the pattern $q$, and for a set of patterns $S$ let $Av_n(S)$ be the set of permutations avoiding all of the patterns in $S$. We say that two sets of patterns $S$ and $S'$ are \textit{Wilf-equivalent} if $|Av_n(S)| = |Av_n(S')|$ for all $n$. We call a Wilf-equivalence \textit{unbalanced} if $S$ and $S'$ are each different sizes.

Much is known about balanced Wilf-equivalences, but less has been said about unbalanced Wilf-equivalences. The first unbalanced Wilf-equivalences between finite sets of patterns were found in \cite{BursteinPantone} and \cite{BloomBurstein}, and more have recently been found in \cite{Defant} and in \cite{BeanNadeauUlfarsson}. Most of these results have been found by using the structure of the permutations in a set $Av_n(S)$ to find an exact enumeration, and comparing this enumeration to the enumeration of another set $Av_n(S')$. Notably, recently in \cite{Nick} some more unbalanced Wilf-equivalences have been found not by direct enumeration but through constructing a bijection between the classes.

In this paper, we find a direct bijection between $Av_n(3412,4321)$ and $Av_n(3412,4231,45321,54312)$. The former class of permutations is equivalent to $Av_n(4123, 4321)$, which  was directly enumerated in \cite{KremerShiu}.  The latter class of permutations has recently been directly enumerated by Azam and Richmond \cite{Richmond}, by considering these permutations as a restriction of smooth permutations and analyzing their structure using staircase diagrams. 

To construct this bijection, we will construct an injection in the forward direction in Section 3 and an injection in the backwards direction in Section $4$. The large part of the work is in proving Lemmas $3.2$ and $4.2$, which give the technical conditions we need to show that the maps in Theorems $3.3$ and $4.3$ are the injections we want. We collect these results to give the main result as Corollary $5.1$.

\section{Preliminaries}

First, we recall some definitions.

\begin{definition}
    We say that a permutation $p$ \textbf{contains} a pattern $q = q_1 ... q_k$ if there is a subsequence $n_1, ..., n_k$ so that $p_{n_i} < p_{n_j}$ if and only if $q_i < q_j$. If $p$ does not contain $q$, then we say $p$ \textbf{avoids} $q$.
\end{definition}

\begin{example}
    The permutation $4213756$ contains a $123$ pattern in the entries $156$ but avoids $4231$.
\end{example}

\begin{definition}
    Say that a sequence of entries $p_{i_1}, ..., p_{i_n}$ of a permutation $p$ is \textbf{(lexicographically) left} of the sequence of entries $p_{j_1}, ..., p_{j_n}$ if $i_1...i_n$ is lexicographically before $j_1...j_n$ (i.e. $i_1< j_1$, or $i_1=j_1$ and $i_2 < j_2$, or $i_1 = j_1$ and $i_2=j_2$ and $i_3<j_3$, etc.). Similarly define \textbf{(lexicographically) right}.
\end{definition}

\begin{example}
    In the permutation $365421$, the entries $321$ are to the left of the entries $654$.
\end{example}

\section{The forwards map}

\begin{definition}
    Let $\phi$ be the map which takes a $4231$ pattern in a permutation with the leftmost $23$ entries and swaps the entries taking the role of $2$ and $3$.
\end{definition}

We prove the following properties hold for $\phi$.

\begin{lemma}
    Let $p$ be a permutation avoiding $3412$, $45321$, and $54312$, and which is such that the rightmost $32$ entries participating in a $4321$ pattern are to the left of the leftmost $23$ entries participating in a $4231$. Then the following are true:
    \begin{enumerate}
        \item $\phi(p)$ avoids $3412$.
        \item $\phi(p)$ sends the leftmost entries participating as the $23$ in a $4231$ pattern to the rightmost entries participating as the $32$ in a $4321$ pattern.
        \item The leftmost $23$ entries in a $4231$ pattern in $\phi(p)$ are to the right of the leftmost $23$ entries in a $4231$ pattern in $p$.
        \item $\phi(p)$ avoids $45321$ and $54312$.
    \end{enumerate}
\end{lemma}

\begin{proof}
\begin{figure}[h]
        \centering
        \begin{tikzpicture}[scale = 1.6, point/.style={circle,draw=black!100,fill=black!100,thick,
                     inner sep=0pt,minimum size=1mm},ghostpoint/.style={circle,draw=black!50,fill=black!50,thick,
                     inner sep=0pt,minimum size=1mm}
                          ]
                          
      \draw[draw = black!80!,fill=gray!20] (0,0) rectangle node[anchor= center]{} (1,1);
      \draw[draw = black!80!,fill=gray!20] (1,0) rectangle node[anchor= center]{} (2,1);
      \draw[draw = black!80!,fill=orange!20] (2,0) rectangle node[anchor= center]{} (3,1);
      \draw[draw = black!80!,fill=red!20] (3,0) rectangle node[anchor= center]{} (4,1);
      \draw[draw = black!80!,fill=red!20] (4,0) rectangle node[anchor= center]{} (5,1);

      \draw[draw = black!80!,fill=gray!20] (0,1) rectangle node[anchor= center]{} (1,2);
      \draw[draw = black!80!,fill=red!20] (1,1) rectangle node[anchor= center]{} (2,2);
      \draw[draw = black!80!,fill=orange!20] (2,1) rectangle node[anchor= center]{} (3,2);
      \draw[draw = black!80!,fill=red!20] (3,1) rectangle node[anchor= center]{} (4,2);
      \draw[draw = black!80!,fill=red!20] (4,1) rectangle node[anchor= center]{} (5,2);

      \draw[draw = black!80!,fill=blue!20] (0,1+1) rectangle node[anchor= center]{} (1,2+1);
      \draw[draw = black!80!,fill=red!20] (1,1+1) rectangle node[anchor= center]{} (2,2+1);
      \draw[draw = black!80!,fill=red!20] (2,1+1) rectangle node[anchor= center]{} (3,2+1);
      \draw[draw = black!80!,fill=blue!20] (3,1+1) rectangle node[anchor= center]{} (4,2+1);
      \draw[draw = black!80!,fill=blue!20] (4,1+1) rectangle node[anchor= center]{} (5,2+1);

      \draw[draw = black!80!,fill=red!20] (0,1+2) rectangle node[anchor= center]{} (1,2+2);
      \draw[draw = black!80!,fill=gray!20] (1,1+2) rectangle node[anchor= center]{} (2,2+2);
      \draw[draw = black!80!,fill=red!20] (2,1+2) rectangle node[anchor= center]{} (3,2+2);
      \draw[draw = black!80!,fill=gray!20] (3,1+2) rectangle node[anchor= center]{} (4,2+2);
      \draw[draw = black!80!,fill=gray!20] (4,1+2) rectangle node[anchor= center]{} (5,2+2);

      \draw[draw = black!80!,fill=red!20] (0,1+3) rectangle node[anchor= center]{} (1,2+3);
      \draw[draw = black!80!,fill=red!20] (1,1+3) rectangle node[anchor= center]{} (2,2+3);
      \draw[draw = black!80!,fill=orange!20] (2,1+3) rectangle node[anchor= center]{} (3,2+3);
      \draw[draw = black!80!,fill=gray!20] (3,1+3) rectangle node[anchor= center]{} (4,2+3);
      \draw[draw = black!80!,fill=gray!20] (4,1+3) rectangle node[anchor= center]{} (5,2+3);

      \node (pi) at (1,4) [point] [label=below left:$p_1$] {};
      \node (pj) at (2,2) [point] [label=below left:$p_2$] {};
      \node (pj) at (2,3) [ghostpoint] [label=below left:$p_3'$] {};
      \node (pi) at (3,3) [point] [label=below left:$p_3$] {};
      \node (pi) at (3,2) [ghostpoint] [label=below left:$p_2'$] {};
      \node (pj) at (4,1) [point] [label=below left:$p_4$] {};
      
      \node (E) at (-0.5,0.5) [] [label=center:\textbf{E}] {}; 
      \node (D) at (-0.5,1.5) [] [label=center:\textbf{D}] {};  
      \node (C) at (-0.5,2.5) [] [label=center:\textbf{C}] {};  
      \node (B) at (-0.5,3.5) [] [label=center:\textbf{B}] {};  
      \node (A) at (-0.5,4.5) [] [label=center:\textbf{A}] {};  

      \node (1) at (0.5,-0.5) [] [label=center:\textbf{1}] {}; 
      \node (2) at (1.5,-0.5) [] [label=center:\textbf{2}] {};  
      \node (3) at (2.5,-0.5) [] [label=center:\textbf{3}] {};  
      \node (4) at (3.5,-0.5) [] [label=center:\textbf{4}] {};  
      \node (5) at (4.5,-0.5) [] [label=center:\textbf{5}] {}; 

        \end{tikzpicture}
        \label{fig:placeholder}
    \end{figure}

    Let $p_1p_2p_3p_4$ be a $4231$ pattern in $p$ with entries chosen so that $p_2$ is the leftmost possible, then $p_3$ is the leftmost possible given this choice of $p_2$, then $p_1$ is the leftmost possible given this choice of $p_2$ and $p_3$, then $p_4$ is the leftmost possible given this choice of $p_1, p_2$, and $p_3$. The map $\phi$ will move $p_2$ to $p_2'$ and $p_3$ to $p_3'$. The red squares in the grid below indicate where there are no entries in $p$ due to the assumptions on $p$ above.

    Proof of 1:
    \\
    Suppose $\phi(p)$ did have a $3412$ pattern $q=q_1q_2q_3q_4$. This pattern must either include one of $p_2'$ and $p_3'$ and contain an entry in a blue box and an entry in an orange box, or it must contain both $p_2'$ and $p_3'$. This is because if $q$ has only one of $p_2'$ and $p_3'$ and has no entry in a blue box, then replacing that entry by $p_3$ or $p_2$ respectively would give a $3412$ pattern in $p$, which is a contradiction. Likewise if $q$ has only one of $p_2'$ and $p_3'$ and has no entry in an orange box, then replacing that entry by $p_2$ or $p_3$ respectively would give a $3412$ pattern in $p$, which is a contradiction.
    
    \begin{itemize}
        \item Case 1: Suppose first that the $3412$ pattern has only one of the shifted points. Suppose also that $q_1$ is in C1.
        \begin{itemize}
            \item If $p_3'$ is in the $3412$ pattern, then $p_3'$ must be $q_2$. Further, $q_3$ must be in D3 or E3 for an orange box to be included, and $q_4$ can be in D3, E3 (if $q_3$ is), C4, or C5. But this means that $q_1p_1q_3q_4$ is a $3412$ pattern in $p$ which is a contradiction. 
            \item If instead $p_2'$ is in the $3412$ pattern, then $p_2'$ must be either $q_3$ or $q_4$. If $p_2'$ is $q_3$, then $q_4$ must be in C4 or C5. In either case, $p$ then has a $3412$ pattern in $q_1p_1p_2q_4$. If $p_2'$ is instead $q_4$, then for an orange box to be occupied, $q_3$ must be in D3 or E3. If $q_3$ is in D3, then $p$ has a $45321$ pattern in $q_1p_1p_2q_3p_4$, and if $q_3$ is in E3, then $p$ has a $3412$ pattern in $q_1p_1q_3p_4$. 
        \end{itemize}
        \item Case 2: Suppose the $3412$ pattern again has only one of the shifted points but that $q_1$ is not in C1. 
        \begin{itemize}
            \item Suppose first $p_3'$ is in the pattern. In order for both an orange and a blue box to be occupied, we must have that $p_3'$ is $q_1$, $q_2$ is in A3, and $q_4$ must be in C4 or C5. However, if $q_4$ is in C4, then $p$ has a $45321$ pattern in $p_1q_2p_3q_4p_4$. If, on the other hand, $q_4$ is in C5, then $p$ has a $3412$ pattern in $p_1q_2p_4q_4$. 
            \item Now suppose $p_2'$ is in the pattern. Then because a blue box must be occupied and because an entry can only be below and to the right of $p_2'$ if it is in fact $p_4$, we must have that $p_2'$ is $q_3$ and $q_4$ is in C4 or C5. Further, for an orange box to be occupied, we must have that $q_2$ is in A3. If $q_4$ is in C4, then $p$ has a $45321$ pattern in $p_1q_2p_3q_4p_4$. If, on the other hand, $q_4$ is in C5, then $p$ has a $3412$ pattern in $p_1q_2p_4q_4$.
        \end{itemize}
        \item Case 3: Finally, suppose the $3412$ pattern contains both $p_2'$ and $p_3'$. We no longer require that both and orange and a blue box must be occupied by parts of the pattern $q_1q_2q_3q_4$. We split this into 4 cases depending on the parts of the pattern these points can act as.
        \begin{itemize}
            \item First, assume $p_3'$ is $q_1$ and $p_2'$ is $q_3$. Then $q_2$ must be in A3 and $q_4$ must be in C4 or C5.  If $q_4$ is in C4, then $p$ has a $45321$ pattern in $p_1q_2p_3q_4p_4$, and if $q_4$ is in C5, then $p$ has a $3412$ pattern in $p_1q_2p_4q_4$.
            \item Next, assume $p_3'$ is $q_1$ and $p_2'$ is $q_4$. Then $q_2$ must be in A3 and $q_3$ must be in D3 or E3. In either case, $p$ has a $3412$ pattern in $p_1q_2q_3p_3$.
            \item Now, assume $p_3'$ is $q_2$ and $p_2'$ is $q_3$. Then $q_1$ must be in C1 and $q_4$ must be in C4 or C5. In either case, $p$ then has a 3412 pattern in $q_1p_1p_2q_4$.
            \item Finally, assume $p_3'$ is $q_2$ and $p_2'$ is $q_4$. Then $q_1$ must be in C1 and $q_3$ must be in D3 or E3. If $q_3$ is in E3, then $p$ has a $3412$ pattern in $q_1p_1q_3p_4$, and if $q_3$ is in D3, then $p$ has a $45321$ pattern in $q_1p_1p_2q_3p_4$. 
        \end{itemize}
    \end{itemize}  

    \vspace{10pt}

    Proof of 2:
    \\
    Suppose $\phi(p)$ does have a $4321$ pattern $q=q_1q_2q_3q_4$ in which $q_2q_3$ is to the right of $p_3'p_2'$. Again, either $p_2'$ and $p_3'$ are in $q$, or exactly one of $p_2'$ and $p_3'$ is in $q$ and $q$ must contain an entry in a blue box and an entry in an orange box. 

    \begin{itemize}
        \item Case 1: Suppose first only one of the points $p_2'$, $p_3'$ is in $q$.
        \begin{itemize}
            \item Then because $q$ must be entirely decreasing, there is no way for $q$ to have an entry in an orange box and in a blue box. 
        \end{itemize}
        
        \item Case 2: Suppose both of $p_3'$ and $p_2'$ are in $q$. 
        \begin{itemize}
            \item Since there are no entries allowed in C3, then $p_3'$ and $p_2'$ must be consecutive entries of $q$. Further, since $q_2q_3$ is to the right of $p_3'p_2'$, then $p_3'$ and $p_2'$ must be $q_1$ and $q_2$. But there is only one point which can go below and to the right of $p_2'$, namely $p_4$, so there is no valid way to place $q_3$ and $q_4$, so we have a contradiction. 
        \end{itemize}
        
    \end{itemize}

    \vspace{10pt}
    
    Proof of 3: Suppose there is a $4231$ pattern $q=q_1q_2q_3q_4$ in $\phi(p)$ with $q_2q_3$ to the left of $p_2p_3$.
    \begin{itemize}
        \item Case 1: Suppose only one of $p_2'$ and $p_3'$ are in $q$. Then it must be the case that $q$ has an entry in C1 and an entry in A3 or both an entry in D3 or $E3$ and an entry in C4 or C5, as otherwise $\phi(p)$ will have a $3412$ pattern which is forbidden by $(1)$.
        \begin{itemize}
            \item In the former case, then $q_2$ must be in C1 and $q_3$ must be in A3 since $4231$ only has an increase in those positions, but then there is nowhere to place $q_1$ so this is a contradiction. 
            \item In the latter case, then $q_2$ must be in D3 or E3 and $q_3$ must be in C4 or C5. But then $q_2q_3$ is to the right of $p_2p_3$, which contradicts the assumption.
        \end{itemize}
        \item Case 2: If both $p_2'$ and $p_3'$ are in $q$, then in order for $q_2q_3$ to be to the left of $p_2p_3$, $p_3'$ must be $q_2$ or $q_3$. 
        \begin{itemize}
            \item If $p_3'$ is $q_2$, then $q_3$ must go in A3 and there is nowhere to place $q_1$.
            \item If $p_3'$ is $q_3$, then $q_2$ must go in C1 and again there is nowhere to place $q_1$.
        \end{itemize}
    \end{itemize}

    \vspace{10pt}
    
    Proof of 4:
    \\
    Suppose $\phi(p)$ did have a $45321$ pattern $q_1q_2q_3q_4q_5$. This pattern must include one of either $p_2'$ or $p_3'$, and must contain an entry in a blue box and an entry in an orange box, or it must contain both $p_2'$ and $p_3'$.

    \begin{itemize}
        \item Case 1: Suppose first that $q_1$ is in C1. By (1), there can be no entries in D3 or E3 since otherwise $\phi(p)$ would have a $3412$ pattern with $q_1$, $p_1$, this point, and $p_2'$.
        \begin{itemize}
            \item If $p_3'$ is in the pattern, then no orange box can be occupied. 
            \item If instead $p_2'$ is in the pattern, then it must act as $q_5$ with $q_2, q_3,q_4$ in C1 since $q_1$ is. Again, then no orange box can be occupied. 
        \end{itemize}
        \item Case 2: Suppose that the pattern still has exactly one of the shifted points but assume instead that $q_1$ is not in C1. 
        \begin{itemize}
            \item Suppose first that $p_3'$ is part of the $45321$ pattern. Then in order for both orange and blue boxes to contain pieces of the $45321$ pattern, $p_3'$ must be $q_1$. Further, A3 must contain $q_2$ and the blue boxes must contain $q_3,q_4,q_5$. But then replacing $p_3'$ by $p_1$ gives a $45321$ pattern in $p$, a contradiction.
            \item Now suppose instead that $p_2'$ is part of the $45321$ pattern. Then $p_2'$ must be $q_5$, and $q_1,q_2,q_3,q_4$ must be in B2 and/or A3. Thus, there are no parts of the pattern in a blue box, a contradiction. 
        \end{itemize}
        \item Case 3: Finally, assume both $p_2'$ and $p_3'$ are in the $45321$ pattern. Because no entries are allowed in C3 or anywhere below and to the right of $p_2'$ (except $p_4$), these points must act as $q_3$ and $q_4$ with $p_4$ as $q_5$, or $q_4$ and $q_5$.
        \begin{itemize}
            \item If $p_3'$ and $p_2'$ are $q_3$ and $q_4$ with $p_4$ as $q_5$, then $q_1$ and $q_2$ must both be in B2. Then $p$ has a $3412$ pattern in $q_1q_2p_2p_3$.
            \item If instead these points act as $q_4$ and $q_5$, then $q_1, q_2, q_3$ must all be in B2 and we have the same issue as above. 
        \end{itemize}
    \end{itemize}

    Now, suppose $\phi(p)$ has a $54312$ pattern $q_1q_2q_3q_4q_5$. Again, this pattern must contain one of $q_2'$ and $q_3'$ and an entry in an orange and a blue box, or both of the shifted points. 

    \begin{itemize}
        \item Case 1: Suppose first that $q_1$ is in C1 and exactly one of the shifted points is included. Then again because of (1), we can't have any entries in D3 or E3. Also, $p_3'$ can't be in the pattern, so $p_2'$ must be. Further, it must be $q_5$, otherwise there are no pieces in an orange box, and C1 must have $q_1,q_2,q_3$ in a 321 pattern. However, since D3 and E3 are blocked, there's nowhere to place $q_4$. 
        \item Case 2: Now suppose $q_1$ is not in C1 but we still have exactly one of the shifted points included.
        \begin{itemize}
            \item Suppose first $p_3'$ is in the $54312$ pattern. For a blue box to be populated by the pattern, $p_3'$ must be $q_3$. Then $p$ has a $54312$ pattern in $p_1q_1q_2q_4q_5$, a contradiction. 
            \item If instead $p_2'$ is in the pattern, then for a blue box to be populated, $p_2'$ must be $q_4$, with $q_1q_2q_3$ in B2 or A3 and $q_5$ in C4 or C5. In order for an orange box to be populated, it must be true that $q_1q_2q_3$ is in A3. If $q_5$ is in C4, then $p_1q_1p_3q_5p_4$ forms a $45321$ pattern in $p$, and if $q_5$ is in C5 then $q_1q_2q_3p_4q_5$ forms a $54312$ pattern in $p$. 
        \end{itemize}
        \item Case 3: Finally, suppose both $p_2'$ and $p_3'$ are in the $54312$ pattern. Then $p_3'$ must be $q_3$ and $p_2'$ can be $q_4$ or $q_5$. In both cases, $q_1$ and $q_2$ must be in B2 in a 21 pattern. Then $p$ has a $54312$ pattern in $p_1q_1q_2p_2p_3$.
    \end{itemize}

\end{proof}

Now, to construct the map that we want, we will apply $\phi$ iteratively to the class of permutations in question. The properties proven in the Lemma above will show that this new map is the injection we want.

\begin{theorem}
    There is an injection 
    \[
    \Phi:Av_n(3412,4321) \rightarrow Av_n(3412,4231,45321,54312)
    \]
\end{theorem}

\begin{proof}
    Let $\Phi$ be the map given by applying $\phi$ iteratively to a permutation $p\in Av_n(3412,4321)$ until it avoids $4231$. Properties $1,3,$ and $4$ allow us to apply $\phi$ iteratively and maintain that properties $1,2,3,$ and $4$ are true in the image. Since $\phi$ will always turn a $4231$ pattern into a $4321$ pattern, then $\phi$ will always make a permutation strictly lexicographically larger, hence $\Phi$ must eventually terminate.

    Notice that $\Phi$ does map into $Av_n(3412,4231,45321,54312)$ by properties $1$ and $3$ of the above lemma.

    To see that $\Phi$ is an injection, note that we can start with $\Phi(p)$ and take the rightmost $4321$ pattern and swap the $3$ and the $2$ entry, and do the same with the new rightmost $4321$ pattern, and repeat until the permutation avoids $4321$. By property $2$ of the above lemma, this will give us back $p$. 
\end{proof}

\section{The backwards map}

\begin{definition}
    Let $\psi$ be the map which takes a copy of $4321$ in a permutation with the rightmost $32$ entries  and swaps the entries taking the role of $3$ and $2$.
\end{definition}

We prove the following properties hold for $\psi$.

\begin{lemma}
    Let $p$ be a permutation avoiding $3412$, $45321$, and $54312$, and which is such that the rightmost $32$ entries participating in a $4321$ pattern are  to the left of the leftmost $23$ entries participating in a $4231$ pattern. Then the following are true:
    \begin{enumerate}
        \item $\psi(p)$ avoids $3412$.
        \item $\psi(p)$ sends the rightmost entries participating as the $32$ in a $4321$ pattern to the leftmost entries participating as the $23$ in a $4231$ pattern.
        \item The rightmost $32$ entries in a $4321$ pattern in $\psi(p)$ are to the left of the rightmost $32$ entries in a $4321$ pattern in $p$.
        \item $\psi(p)$ avoids $45321$ and $54312$.
    \end{enumerate}
\end{lemma}

\begin{proof}
    Let $p_1p_3p_2p_4$ be the $4321$ pattern in $p$ with entries chosen so that $p_2$ is rightmost, then $p_3$ is rightmost given this choice of $p_2$, then $p_1$ is rightmost given this choice of $p_2$ and $p_3$, and then $p_4$ is rightmost given this choice of $p_1,p_2$, and $p_3$. The map $\psi$ will move $p_2$ to $p_2'$ and $p_3$ to $p_3'$. The red squares in the following figure indicate where it is impossible for there to be any entry in $p$ without contradicting one of the assumptions.
    
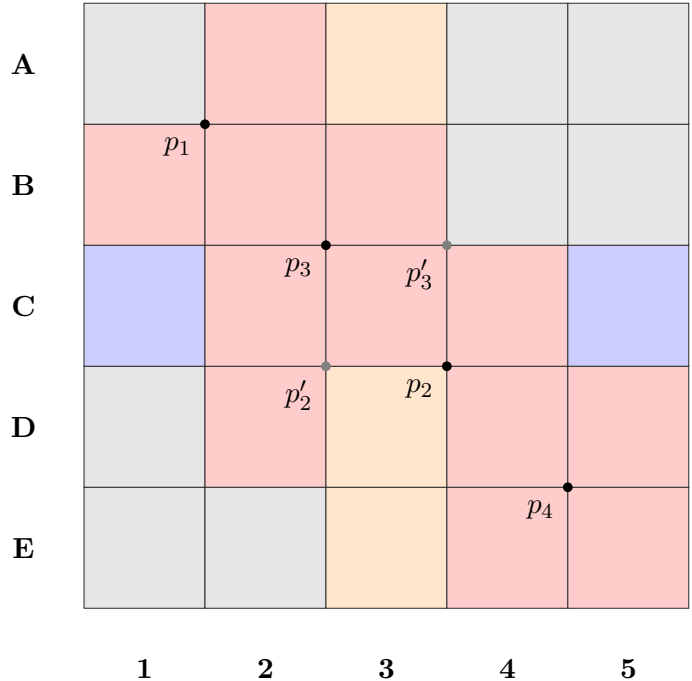
\begin{figure}[h]
        \centering
        \begin{tikzpicture}[scale = 1.6, point/.style={circle,draw=black!100,fill=black!100,thick,
                     inner sep=0pt,minimum size=1mm},ghostpoint/.style={circle,draw=black!50,fill=black!50,thick,
                     inner sep=0pt,minimum size=1mm}
                          ]
                          
      \draw[draw = black!80!,fill=gray!20] (0,0) rectangle node[anchor= center]{} (1,1);
      \draw[draw = black!80!,fill=gray!20] (1,0) rectangle node[anchor= center]{} (2,1);
      \draw[draw = black!80!,fill=orange!20] (2,0) rectangle node[anchor= center]{} (3,1);
      \draw[draw = black!80!,fill=red!20] (3,0) rectangle node[anchor= center]{} (4,1);
      \draw[draw = black!80!,fill=red!20] (4,0) rectangle node[anchor= center]{} (5,1);

      \draw[draw = black!80!,fill=gray!20] (0,1) rectangle node[anchor= center]{} (1,2);
      \draw[draw = black!80!,fill=red!20] (1,1) rectangle node[anchor= center]{} (2,2);
      \draw[draw = black!80!,fill=orange!20] (2,1) rectangle node[anchor= center]{} (3,2);
      \draw[draw = black!80!,fill=red!20] (3,1) rectangle node[anchor= center]{} (4,2);
      \draw[draw = black!80!,fill=red!20] (4,1) rectangle node[anchor= center]{} (5,2);

      \draw[draw = black!80!,fill=blue!20] (0,1+1) rectangle node[anchor= center]{} (1,2+1);
      \draw[draw = black!80!,fill=red!20] (1,1+1) rectangle node[anchor= center]{} (2,2+1);
      \draw[draw = black!80!,fill=red!20] (2,1+1) rectangle node[anchor= center]{} (3,2+1);
      \draw[draw = black!80!,fill=red!20] (3,1+1) rectangle node[anchor= center]{} (4,2+1);
      \draw[draw = black!80!,fill=blue!20] (4,1+1) rectangle node[anchor= center]{} (5,2+1);

      \draw[draw = black!80!,fill=red!20] (0,1+2) rectangle node[anchor= center]{} (1,2+2);
      \draw[draw = black!80!,fill=red!20] (1,1+2) rectangle node[anchor= center]{} (2,2+2);
      \draw[draw = black!80!,fill=red!20] (2,1+2) rectangle node[anchor= center]{} (3,2+2);
      \draw[draw = black!80!,fill=gray!20] (3,1+2) rectangle node[anchor= center]{} (4,2+2);
      \draw[draw = black!80!,fill=gray!20] (4,1+2) rectangle node[anchor= center]{} (5,2+2);

      \draw[draw = black!80!,fill=gray!20] (0,1+3) rectangle node[anchor= center]{} (1,2+3);
      \draw[draw = black!80!,fill=red!20] (1,1+3) rectangle node[anchor= center]{} (2,2+3);
      \draw[draw = black!80!,fill=orange!20] (2,1+3) rectangle node[anchor= center]{} (3,2+3);
      \draw[draw = black!80!,fill=gray!20] (3,1+3) rectangle node[anchor= center]{} (4,2+3);
      \draw[draw = black!80!,fill=gray!20] (4,1+3) rectangle node[anchor= center]{} (5,2+3);

      \node (pi) at (1,4) [point] [label=below left:$p_1$] {};
      \node (pj) at (2,2) [ghostpoint] [label=below left:$p_2'$] {};
      \node (pj) at (2,3) [point] [label=below left:$p_3$] {};
      \node (pi) at (3,3) [ghostpoint] [label=below left:$p_3'$] {};
      \node (pi) at (3,2) [point] [label=below left:$p_2$] {};
      \node (pj) at (4,1) [point] [label=below left:$p_4$] {};
      
      \node (E) at (-0.5,0.5) [] [label=center:\textbf{E}] {}; 
      \node (D) at (-0.5,1.5) [] [label=center:\textbf{D}] {};  
      \node (C) at (-0.5,2.5) [] [label=center:\textbf{C}] {};  
      \node (B) at (-0.5,3.5) [] [label=center:\textbf{B}] {};  
      \node (A) at (-0.5,4.5) [] [label=center:\textbf{A}] {};  

      \node (1) at (0.5,-0.5) [] [label=center:\textbf{1}] {}; 
      \node (2) at (1.5,-0.5) [] [label=center:\textbf{2}] {};  
      \node (3) at (2.5,-0.5) [] [label=center:\textbf{3}] {};  
      \node (4) at (3.5,-0.5) [] [label=center:\textbf{4}] {};  
      \node (5) at (4.5,-0.5) [] [label=center:\textbf{5}] {}; 

        \end{tikzpicture}
        \label{fig:placeholder}
    \end{figure}
    \\
    Similarly to the proof of Lemma 2.2, note that for a pattern $q$ to be in $\psi(p)$ but not in $p$, then $q$ must either have one of $p_2'$ and $p_3'$ and an entry in an orange box and an entry in a blue box, or must have both of $p_2'$ and $p_3'$. 
    
    However, before beginning to prove each piece, we make the following observations.
    \begin{itemize}
        \item Suppose this (so far unspecified) pattern $q$ contains only one of $p_2'$ and $p_3'$. Suppose first that $q$ has an entry in C1. Then notice that it may not have an entry in D3 or E3 since this would create a $3412$ pattern in $p$ with the existing $p_1$ and $p_2$. So, the only orange box available is A3.
        \item Suppose $q$ again only has one of the shifted points, but now suppose it contains an entry in C5. Then it may not have an entry in A3 since this would create a $3412$ pattern in $p$ with $p_3$ and $p_2$. Thus, the only blue boxes available are D3 and E3. 
        \item From the two items above, we need only to check the cases in which $q$ has an entry in both C1 and A3 or an entry in D3 or E3 and one in C5.  
        \item Now suppose the only increases in this pattern $q$ are consecutive in position and value. Then in both of the above cases, it is impossible for $q$ to contain either $p_2'$ or $p_3'$: 
        \begin{itemize}
            \item Indeed, if $p_2'$ is in $q$, then the increase from C1 to A3 is not consecutive in position, and the increase from D3 or E3 to C5 is not consecutive in value.
            \item If instead $p_3'$ is in $q$, then the increase from C1 to A3 is not consecutive in value, and the increase from D3 or E3 to C5 is not consecutive in position.
        \end{itemize} 
        So for such $q$, it must be the case that $q$ contains both $p_2'$ and $p_3'$ and thus need not have an entry both in a blue box and in an orange box.
        \item Finally, note that each of the patterns $\{3412, 4231, 4321, 45321, 54312\}$ fall into the category described above. Thus, we may use the preceding observation on each part of the proof.
    \end{itemize}

    Proof of 1:
    Suppose $\psi(p)$ does have a $3412$ pattern $q= q_1q_2q_3q_4$.

    Then $q$ must contain both of $p_2'$ and $p_3'$. If $p_2'$ and $p_3'$ are $q_1$ and $q_2$, then we have a contradiction, since there is only one possible entry to the right of $p_3'$ and smaller than $p_2'$. If instead $p_2'$ and $p_3'$ are $q_3$ and $q_4$, then $q_1$ and $q_2$ must be in A1 or one of them may be $p_1$. In any of these cases, then $q_1q_2p_3p_2p_4$ forms a $45321$ pattern in $p$, which is a contradiction.
    
    \vspace{10pt}

    Proof of 2:
    Suppose $\psi(p)$ does have a $4231$ pattern $q = q_1q_2q_3q_4$ with $q_2q_3$ to the left of $p_2'p_3'$.

    Again, $q$ must contain both of $p_2'$ and $p_3'$. But the only way for $q$ to contain $p_2'$ and $p_3'$ is as $q_2q_3$ (since $4231$ has only one increase). This contradicts that $q_2q_3$ is to the left of $p_2'p_3'$.

    \vspace{10pt}

    Proof of 3:
    Suppose $\psi(p)$ does have a $4321$ pattern $q=q_1q_2q_3q_4$ with $q_2q_3$ to the right of $p_2'p_3'$. 

    Then $q$ must contain both of $p_2'$ and $p_3'$. But this is impossible, since, $q$ is strictly decreasing, so we have a contradiction.
    
    \vspace{10pt}

    Proof of 4:
    Suppose $\psi(p)$ does have a $45321$ pattern $q = q_1q_2q_3q_4q_5$.

    We know $q$ must contain both of $p_2'$ and $p_3'$. The only way for $q$ to contain $p_2'$ and $p_3'$ is as $q_1q_2$. But there is only one possible entry to the right of $p_3'$ and smaller than $p_2'$, so we have a contradiction.

    Suppose instead that $\psi(p)$ has a $54312$ pattern $q=  q_1q_2q_3q_4q_5$.

    Again, $q$ must contain both of $p_2'$ and $p_3'$. The only way for $q$ to contain $p_2'$ and $p_3'$ is as $q_4q_5$. So $q_1q_2q_3$ are all either in A1 or possibly are $p_1$. In all cases, then $q_1q_2p_3p_2p_4$ forms a $45321$ pattern in $p$, which is a contradiction.

\end{proof}

Now, to construct the map that we want, we will apply $\psi$ iteratively to the class of permutations in question. The properties proven in the Lemma above will show that this new map is the injection we want.

\begin{theorem}
    There is an injection 
    \[
    \Psi: Av_n(3412,4231,45321,54312) \rightarrow Av_n(3412,4321)
    \]
\end{theorem}

\begin{proof}
    Let $\Psi$ be the map given by applying $\psi$ iteratively to a permutation $p\in Av_n(3412,4231,45321,54312)$ until it avoids $4321$. Properties $1,3,$ and $4$ allow us to apply $\psi$ iteratively and maintain that properties $1,2,3,$ and $4$ are true in the image. Since $\psi$ will always turn a $4321$ pattern into a $4231$ pattern, then $\psi$ will always make a permutation strictly lexicographically smaller, hence $\Psi$ must eventually terminate.

    Notice that $\Psi$ does map into $Av_n(3412,4321)$ by properties $1$ and $3$ of the above lemma.

    To see that $\Psi$ is an injection, note that we can start with $\Psi(p)$ and swap the leftmost entries participating as the $23$ in a $4231$ pattern, and do the same with the new leftmost entries participating as the $23$ in a $4231$ pattern, and repeat until the permutation avoids $4231$. By property $2$ of the above lemma, this will give us back $p$. 
\end{proof}

\section{Conclusion and Further Directions}
Theorems $3.3$ and $4.3$ together give us the following, which is our main result.
\begin{corollary}
For all $n\geq 1$ we have
    \[
    |Av_n(3412,4321)| = |Av_n(3412,4231,45321,54312)|.
    \]
\end{corollary} 
What is nice about the method we have used is that $3412$ is not used in the main definition of the maps. In \cite{BursteinPantone} it is shown that $|Av_n(4321)| = |Av_n(4231, 5276143)|$ for all $n\geq 1$. The permutation classes in our Corollary $5.1$ are exactly the classes obtained by taking the classes in the paper by Burstein and Pantone and restricting further each side to avoid $3412$. Then it is perhaps a good guess that our map is in fact the restriction of some larger map on these classes, and in fact some preliminary computational evidence supports this guess. This would be interesting, as then perhaps this larger map could be restricted in different ways to obtain new unbalanced Wilf-equivalences.
\begin{question}
Is there a bijection
\[
F: Av_n(4321) \rightarrow Av_n(4231, 5276143)
\]
which acts in the same way as our map $\Phi$ does when restricted to $3412$-avoiders? Further, can this map be restricted in other ways to obtain other Wilf-equivalences?
\end{question}
Question $5.2$ will be addressed soon in a further paper.
\\
Something else worth noting is that there are other conjectured unbalanced Wilf-equivalences where the kind of swapping we have done seems natural. For instance, in \cite{BursteinPantone} is is conjectured for all $n\geq 1$ that $|Av_n(2413)| = |Av_n(2143, 246135)|$. Computational evidence demonstrates that using exactly the same kind of map as in this paper but for these classes (iteratively swapping middle entries in leftmost $2143$ patterns) will not work, but it may be true that some modification of this map will work.
\begin{question}
    Are there any other unbalanced Wilf-equivalences, for instance the conjectured \cite{BursteinPantone} $|Av_n(2413)| = |Av_n(2143, 246135)|$, that can be proven by a similar method to ours?
\end{question}


\end{document}